\documentclass{amsart}
\usepackage{amssymb}
\usepackage{enumerate}

\usepackage{color,soul,xspace}
\definecolor{paleblue}{rgb}{0.7804,0.8353,0.9725}
\sethlcolor{paleblue}

\theoremstyle{plain}
\newtheorem{thm}{Theorem}[section]

\theoremstyle{definition}

\theoremstyle{remark}

\newtheorem{claim}[thm]{Claim}

\numberwithin{equation}{section}

\newcommand{\zj}{\emptyset'}
\renewcommand{\epsilon}{\varepsilon}
\renewcommand{\phi}{\varphi}
\renewcommand{\succeq}{\succcurlyeq}
\renewcommand{\preceq}{\preccurlyeq}

\newcommand{\seq}[1]{\langle{#1}\rangle}

\DeclareMathOperator{\uh}{\upharpoonright}

\newcommand{\PP}{\mathbb{P}}

\begin{document}

\title{Density, forcing, and the covering problem}

\renewcommand{\datename}{Last compilation:}
\date{\today.\\\indent Last time the following date was changed: April 8, 2013}

\author{Adam R.~Day}
\address{Department of Mathematics\\
University of California, Berkeley\\
Berkeley, CA 94720-3840, USA}
\email{adam.day@math.berkeley.edu}

\author{Joseph S.~Miller}
\address{Department of Mathematics\\
University of Wisconsin\\
Madison, WI 53706-1388, USA}
\email{jmiller@math.wisc.edu}

\thanks{The first author was supported by a Miller Research Fellowship in the Department of Mathematics at the University of California, Berkeley. The second author was supported by the National Science Foundation under grant DMS-1001847.}

\makeatletter
\@namedef{subjclassname@2010}{\textup{2010} Mathematics Subject Classification}
\makeatother
\subjclass[2010]{Primary 03D32; Secondary 68Q30, 03D30}


\maketitle

\begin{abstract}
We present a notion of forcing that can be used, in conjunction with other results, to show that there is a Martin-L\"of random set $X$ such that $X \not \geq_T \zj$ and $X$ computes every $K$-trivial set.
\end{abstract}

\section{Introduction}

Hirschfeldt, Nies and Stephan \cite{HiNiSt:07} proved that if $A\in 2^\omega$ is c.e.\ and there is a Martin-L\"of random $X\geq_T A$ that does not computes $\emptyset'$, then $A$ is $K$-trivial. Stephan asked if this gives a characterization of the c.e.\ $K$-trivial sets. Each $K$-trivial is computable from a c.e.\ $K$-trivial, so this amounts to asking:
\begin{quote}
If $A$ is $K$-trivial, is there a Martin-L\"of random $X\ngeq_T \emptyset'$ that computes $A$?	
\end{quote}
The history and significance of this question, known as the \emph{covering problem}, is presented in a summary paper by the authors of this paper and Bienvenu, Greenberg, Ku\v{c}era, Nies and Turetsky~\cite{everyone}. The present paper, combined with theorems of Bienvenu, Greenberg, Ku\v{c}era, Nies and Turetsky \cite{BiGrKuNiTu:}, and Bienvenu, H\"olzl, Miller and Nies \cite{BiHoMiNi:12}, gives a strong affirmative answer to the covering problem:
\begin{itemize}
	\item[(a)] There is a Martin-L\"of random $X\ngeq_T \emptyset'$ that computes every $K$-trivial.
\end{itemize}
Furthermore, we get two interesting refinements:
\begin{itemize}
	\item[(b)] There is a Martin-L\"of random $X <_T \emptyset'$ that computes every $K$-trivial.
	\item[(c)] If $\seq{A_n\colon n\in\omega}$ is a countable sequence of non-$K$-trivial sets, then there is a Martin-L\"of random $X\ngeq_T \emptyset'$ that computes every $K$-trivial but no $A_n$.
\end{itemize}
By (c), for example, there is an incomplete Martin-L\"of random set $X$ such that the $\Delta^0_2$ sets computed by $X$ are precisely the $K$-trivial sets. This $X$ and Chaitin's $\Omega$ are Martin-L\"of random sets that form an exact pair for the ideal of $K$-trivial sets (i.e., $A \leq_T X, \Omega$ if and only if $A$ is a $K$-trivial set). 

Our contribution to the solution of the covering problem comes out of a careful analysis of Lebesgue density for $\Pi^0_1$ classes. Let $\mu$ be the uniform measure on Cantor space. If $\tau \in 2^{<\omega}$ and $P$ is a measurable set in Cantor space, then we define \[\mu_{\tau}(P) = \frac{\mu(P \cap [\tau])}{ \mu([\tau])}.\]
Given any measurable set $P$ and $X \in 2^\omega$, we define $\rho(P \mid X) = \liminf_{i} \mu_{X\uh i}(P)$. We call $X\in 2^{\omega}$ a \emph{density-one point} if for every $\Pi^0_1$ class $P$ it is the case that
\[
X \in P \implies  \rho(P \mid X) = 1.
\]
If for every $\Pi^0_1$ class $P$ we have $X \in P \implies  \rho(P \mid X) > 1$, then $X$ is called a \emph{positive density point}. In Section~\ref{sect:forcing},  we present a notion of forcing that separates density-one from positive density on the Martin-L\"of random sets. In other words, if $X$ is a sufficiently generic set for this notion of forcing then:
\begin{enumerate}
\item $X$ is Martin-L\"of random, \label{random}
\item $X$ is not a density-one point, \label{not density}
\item $X$ is a positive density point. \label{incomplete}
\end{enumerate}
Properties \eqref{random}, \eqref{not density} and \eqref{incomplete} of generic sets will be established by Claims~\ref{cl:random}, \ref{cl:not density} and \ref{cl:incomplete}, respectively. This forcing notion, in conjunction with the following two theorems, provides a solution to the covering problem.

\begin{thm}[Bienvenu, H\"olzl, Miller and Nies \cite{BiHoMiNi:,BiHoMiNi:12}]
\label{thm:incomplete}
If $X \geq_T \emptyset'$ and Martin-L\"of random, then there exists a $\Pi^0_1$  class $P$ such that $X \in P$ and $\rho(P \mid X) =0$. 
\end{thm}

\noindent We should note that Bienvenu, et al.\ prove Theorem~\ref{thm:incomplete} for density on the unit interval. However, the Cantor space version follows immediately from the proof given in \cite[Theorem 20]{BiHoMiNi:12}.

\begin{thm}[Bienvenu, Greenberg, Ku\v{c}era, Nies and Turetsky~\cite{BiGrKuNiTu:}]
\label{thm:them}
If $X\in 2^\omega$ is Martin-L\"of random and not a density-one point, then $X$ computes every $K$-trivial set.
\end{thm}

\noindent The original proof of Theorem~\ref{thm:them}, given in \cite{BiGrKuNiTu:}, involves several steps. A direct proof, though one relying on more of the theory of $K$-triviality, is given by Bienvenu, H\"olzl, Miller and Nies \cite{BiHoMiNi:}.

By Theorem~\ref{thm:incomplete}, properties \eqref{random} and \eqref{incomplete} imply that $X$ does not compute $\zj$. By Theorem~\ref{thm:them}, properties \eqref{random} and \eqref{not density} imply that $X$ computes all $K$-trivial sets. This shows (a). In Claim~\ref{cl:avoidance}, we show that if $A$ is not $K$-trivial and $X$ is sufficiently generic for our notion of forcing, then $X\ngeq_T A$. This gives us (c); in a sense, our forcing notion is perfectly tuned to constructing incomplete Martin-L\"{o}f random sets that compute all $K$-trivial sets. To show (b), we effectivize the forcing notion in Section~\ref{sec:delta2} to show that there is a $\Delta^0_2$ set $X$ with properties \eqref{random}, \eqref{not density} and \eqref{incomplete}.

\section{The forcing notion}
\label{sect:forcing}

Fix a nonempty $\Pi^0_1$ class $P\subseteq 2^\omega$ that contains only Martin-L\"of random sets. Our forcing partial order $\PP$ consists of conditions of the form $\seq{\sigma,Q}$, where
\begin{itemize}
	\item $\sigma\in 2^{<\omega}$,
	\item $Q\subseteq P$ is a $\Pi^0_1$ class,
	\item $[\sigma]\cap Q\neq\emptyset$,
	\item There is a $\delta<1/2$ such that $(\forall \rho\succeq\sigma)\; [\rho]\cap Q\neq\emptyset\implies \mu_{\rho}(Q)+\delta\geq \mu_{\rho}(P)$.
\end{itemize}
We say that $\seq{\tau,R}$ extends $\seq{\sigma,Q}$ if $\tau\succeq\sigma$ and $R\subseteq Q$. Let $\lambda$ be the empty string. Note that $\seq{\lambda,P}\in\PP$, with $\delta=0$, so $\PP$ is nonempty.

If $G\subseteq\PP$ is a filter, let $X_G = \bigcup_{\seq{\sigma,Q}\in G} \sigma$. In general, $X_G\in 2^{\leq\omega}$. The following claim is trivial to verify and it establishes that if $G$ is sufficiently generic, then $X_G$ is infinite and, in fact, a Martin-L\"of random set.

\begin{claim}\label{cl:random}\ 
\begin{enumerate}
	\item If $\seq{\sigma,Q}\in\PP$ and $\tau\succeq\sigma$ is such that $[\tau]\cap Q\neq\emptyset$, then $\seq{\tau,Q}\in\PP$.
	\item If $G\subseteq\PP$ is sufficiently generic, then $X_G\in P$ (hence it is a Martin-L\"of random set).
\end{enumerate}
\end{claim}


\begin{claim}
\label{cl:not density}
If $G\subseteq\PP$ is sufficiently generic, then $\rho(P\mid X_G)\leq 1/2$, so $X_G$ is not a density-one point.
\end{claim}
\begin{proof}
Fix $n$. We will show that the conditions forcing
\begin{equation}\label{eq:force}
(\exists l\geq n)\;\mu_{X_{\dot{G}} \uh l}(P) < 1/2
\end{equation}
are dense in $\PP$.
Let $\seq{\sigma, Q}$ be any condition and let $\delta$ witness that $\seq{\sigma, Q} \in \PP$. Take $m$ such that $2^{-m} < 1/2-\delta$.
Let $Z$ be the left-most path of $[\sigma] \cap Q$. The set $Z$ is Martin-L\"of random and consequently contains arbitrarily long intervals of $1$'s. Take $\tau \succeq \sigma$ such that $\tau 1^m \prec Z$ and $|\tau| \geq n$.  Because $Z$ is the left-most path in $Q$ it follows that $\mu_{\tau}(Q) \leq 2^{-m}$ and so 
\[
\mu_{\tau}(P) \leq \mu_{\tau}(Q) +\delta < 2^{-m} + \delta < 1/2.
\]
Hence the condition $\seq{\tau, Q}$ extends $\seq{\sigma, Q}$ and forces \eqref{eq:force}.
\end{proof}

\begin{claim}
\label{cl:incomplete}
Let $S\subseteq 2^\omega$ be a $\Pi^0_1$ class and let $\seq{\sigma,Q}\in\PP$. There is an $\epsilon>0$ and a condition $\seq{\tau,R}$ extending $\seq{\sigma,Q}$ such that either
\begin{itemize}
	\item $[\tau]\cap S=\emptyset$, or
	\item If $X\in R$, then $\rho(S\mid X)\geq \epsilon$.
\end{itemize}
Therefore, if $G\subseteq\PP$ is sufficiently generic, then $X_G$ is a positive density point. 
\end{claim}
\begin{proof}
If there is a $\tau\succeq \sigma$ such that $[\tau]\cap S = \emptyset$ and $[\tau]\cap Q\neq \emptyset$, then let $\seq{\tau, Q}$ be our condition. 

Otherwise, it follows that $S \cap [\sigma] \supseteq Q \cap [\sigma]$. In this case let $\delta$ witness that  $\seq{\sigma,Q}\in\PP$. Take $\epsilon$ to be a rational  greater than $0$  and less than $\min\{1/2-\delta,  \mu_{\sigma}(Q)\}$. (Note that $\mu_{\sigma}(Q) >0$ because $[\sigma] \cap Q$ is a non-empty $\Pi^0_1$ class containing only Martin-L\"of random sets.) Consider the $\Pi^0_1$ class
\[
Q^\epsilon_\sigma = \{X\in Q\cap[\sigma] \colon (\forall n\geq|\sigma|)\; \mu_{X\uh n}(Q) \geq \epsilon\}.
\]
We will show that $\seq{\sigma, Q^\epsilon_\sigma}$ is the required condition.

Let $M$ be the set of minimal strings in $\{\rho\succeq\sigma \colon \mu_{\rho}(Q)<\epsilon\}$. Then $M$ is prefix-free and $Q^\epsilon_\sigma = Q\cap[\sigma]\smallsetminus Q\cap[M]$. Summing over $M$ gives us $\mu_{\sigma}(Q\cap [M])<\epsilon$. Hence $\mu_{\sigma}(Q^\epsilon_\sigma) > \mu_{\sigma}(Q) - \epsilon > 0$. This proves that $[\sigma] \cap Q^\epsilon_\sigma \neq \emptyset$.

If $\tau\succeq\sigma$ and $[\tau]\cap Q^\epsilon_\sigma\neq \emptyset$, we can use the same argument to show that $\mu_{\tau}(Q^\epsilon_\sigma) > \mu_{\tau}(Q) - \epsilon$. Because $[\tau]\cap Q \neq \emptyset$,
\[
\mu_{\tau}(P) \leq \mu_{\tau}(Q) + \delta < \mu_{\tau}(Q^\epsilon_\sigma) + \epsilon + \delta.
\]
Hence $\epsilon+\delta < 1/2$ witnesses that $\seq{\sigma, Q^\epsilon_\sigma}$ is a condition.

Note that if $X\in Q^\epsilon_\sigma$, then $\rho(Q \mid X)\geq \epsilon$. This implies that $\rho(S \mid X) \geq \epsilon$ because $S\cap [\sigma] \supseteq Q\cap [\sigma]$, proving the claim.
\end{proof}

A difference test is a $\Pi^0_1$ class $R$ and a uniform sequence of open sets $\seq{U_n \colon n\in \omega}$ such that for all $n$, $\mu(U_n \cap R) \le 2^{-n}$. A set $X$ is captured by such a difference test if $X \in \bigcap_{n\in\omega} U_n \cap R$. We call  a set $X$  \emph{difference random} if it is not captured by any difference test. Difference randomness was introduced by Franklin and Ng~\cite{Fran_Ng_2011}. They showed that $X$ is difference random if and only if $X$ is Martin-L\"of random and $X\not \ge_T \zj$. Hence  Claims \ref{cl:random} and \ref{cl:incomplete} along with Theorem~\ref{thm:incomplete} establish that if $G \subseteq \PP$ is sufficiently generic, then $X_G$ is difference random.

\begin{claim}
\label{cl:avoidance}
Assume that $A\in 2^\omega$ is not $K$-trivial, $\seq{\sigma, Q}\in\PP$, and $\Phi$ is a Turing functional. There is a $\tau\in 2^{<\omega}$ such that $\seq{\tau, Q}$ extends $\seq{\sigma, Q}$ and
\[
(\forall X \in [\tau] \cap Q)[\;\Phi^X = A \implies X \text{ is not difference random}\;].
\]
Therefore, if $G\subseteq \PP$ is sufficiently generic relative to $A$, then $X_G$ does not compute~$A$. 
\end{claim}
\begin{proof}
If there is a $\rho\succeq\sigma$ and an $n$ such that $\Phi^\rho(n)\downarrow \neq A(n)$ and $[\rho]\cap Q \neq \emptyset$, then take $\tau = \rho$.

Assume that no such $\rho$ and $n$ exist. Define $V_n = \{X\in 2^\omega \colon X \in U_n[\Phi^X]\}$, where $U_n[Z]$ is the $n$th level of the universal Martin-L\"of test relative to $Z$. If $X\in V_n \cap [\sigma]\cap Q$, then because $\Phi^X$ is not incompatible with $A$, we have $X\in U_n[\Phi^X]\subseteq U_n[A]$. Hence $\mu(V_n \cap [\sigma] \cap Q) \le \mu U_n(A) \le 2^{-n}$. In other words, $Q$ and $\seq{V_n\cap [\sigma]\colon n\in\omega}$ form a difference test.

Now assume that $X\in [\sigma] \cap Q$ and $\Phi^X = A$. Hirschfeldt, Nies and Stephan \cite{HiNiSt:07} showed that because $A$ is not $K$-trivial, it is not a \emph{base for randomness}. In other words, no set that is Martin-L\"of random relative to $A$ can compute $A$, so $X$ is not random relative to $A$. Therefore, $X\in U_n[A] = U_n[\Phi^X]$ for all $n$. This shows that $X\in \bigcap_{n\in\omega} V_n \cap [\sigma] \cap Q$, so $X$ is not difference random. Hence the claim is satisfied by taking $\tau = \sigma$.
\end{proof}

\section{Effectivizing the forcing}
\label{sec:delta2}

In this section we give a construction of a $\Delta^0_2$ set with properties \eqref{random}, \eqref{not density} and \eqref{incomplete}. This construction is an effectivization of the forcing approach. It is conceptually similar to Sacks's construction of a $\Delta^0_2$ minimal degree, which can be seen as an effectivization of Spector's minimal degree construction~\cite{Sacks_1963,Spector_1956}. 

\begin{thm}
There is a $\Delta^0_2$ set with properties \eqref{random}, \eqref{not density} and \eqref{incomplete}.
\end{thm}
\begin{proof}
Using $\zj$ as an oracle we will define a sequence of conditions $\seq{ p_i \colon i \in \omega}$ in the partial order $\PP$. If $p_i = \seq{\tau, Q}$ and $p_{i+1}= \seq{\sigma, R}$ we will ensure that $\sigma \succeq \tau$. However we will not require that $R \subseteq Q$. Essentially, our oracle construction can make incorrect guesses as to which $\Pi^0_1$ classes to use, provided that a correct guess is made eventually. We will define $p_s$ at stage $s$ of the construction. Additionally at stage $s$ we will define $a_s$ to be a finite sequence of triples $\seq{Q, \sigma, \epsilon}$ where $Q$ is a $\Pi^0_1$ class, $\sigma\in 2^{<\omega}$, and $\epsilon$ is a rational. The sequence $a_s$ will be used to recover information about previous stages in the construction. We let $l(a_s)$ be the length of the sequence $a_s$ and we define partial functions $Q$, $\sigma$ and $\epsilon$ such that if $e <l(a_s)$ then $\seq{Q(s,e), \sigma(s,e), \epsilon(s,e)}$ is the $e$th element of $a_s$. We shall maintain the following construction invariants for all stages $s$:
\begin{enumerate}[(i)]
\item \label{i:order} If $i <j < l(a_s)$, then $Q(s,j) \subseteq Q(s,i)^{\epsilon(s,i)}_{\sigma(s,i)}$ and $\sigma(s,i)\preceq\sigma(s,j)$.
\item \label{i:condition} If $p_s = \seq{\tau, R}$ and  $i < l(a_s)$ then $R \subseteq Q(s,i)^{\epsilon(s,i)}_{\sigma(s,i)}$ and $\sigma(s,i)\preceq\tau$.
\end{enumerate}
The construction is as follows.  Let $\seq{S_e \colon e \in \omega}$ enumerate all $\Pi^0_1$ classes. At stage $0$, let $p_0 = \seq{\lambda, P}$ and let $a_0$ be the empty sequence. Our construction invariants hold trivially.

At stage $s+1$, given $p_s = \seq{\tau, Q}$, we use $\zj$ to find a condition $\seq{\sigma, Q}$ such that $\sigma$ is a strict extension of $\tau$, and $\mu_{\sigma}(P) < 1/2$. Claim \ref{cl:not density}  established that such a condition exists, and as the value of $\mu_{\sigma}(P)$ is computable in $\zj$  we can simply search for a suitable $\sigma$. At this point we ask the following question. Does there exist  $e < l(a_s)$ and $\nu$ such that 
\begin{equation}
\label{eq:test}
(\tau \preceq \nu \preceq \sigma) \wedge ([\tau] \cap S_e \ne \emptyset)\wedge(\mu_{\nu}(S_e) < \epsilon(s,e))?
\end{equation}
If not, then we define $p_{s+1} = \seq{\sigma,  Q^{\epsilon_{s+1}}_\sigma}$ where $\epsilon_{s+1}$ is chosen to make $p_{s+1}$ a condition. It follows from the proof of Claim \ref{cl:incomplete} that $\epsilon_{s+1}$ can simply be chosen to be strictly less than $\min\{\mu_{\sigma}(Q), 1/2 - \sum_{i\le s} \epsilon_i\}$. Define $a_{s+1}$ to be the sequence obtained by appending $\seq{Q, \sigma, \epsilon_{s+1}}$ to the end of $a_s$.  Note that the construction invariants are maintained. 

If \eqref{eq:test} holds for some suitable  $e$ and $\nu$, then choose some $e$ and $\nu$ such that $e$ is minimal.  
Our construction invariants ensure that $Q \subseteq Q(s,e)^{\epsilon(s,e)}_{\sigma(s,e)}$ and $\nu\succeq\tau\succeq\sigma(s,e)$. 
This  implies that $\mu_{\nu}(Q(s, e)) \ge \epsilon(s,e)$. 
Therefore there is some $\xi \succeq \nu$ such that 
$[\xi] \cap Q(s,e) \ne \emptyset$ and  $[\xi] \cap S_e = \emptyset$. 
Define $p_{s+1} = \seq{\xi, Q(s,e)}$ and define $a_{s+1} = a_s \uh e$. Observe that construction invariant \eqref{i:order} is maintained because $a_{s+1}$ is a subsequence of $a_s$, and construction invariant \eqref{i:condition} is maintained because construction invariant \eqref{i:order}  held at stage $s$. 
This ends the construction. 

Let $X = \bigcup\, \{ \tau \colon (\exists s, Q)\; p_s=\seq{\tau, Q}\}$. To verify that $X$ has the desired properties, we first show that $\lim_s l(a_s) = \infty$. Assume that for some $s_0$, for all $s \ge s_0$, $l(a_s) \ge e$. Assume at some stage $s_1 > s_0$, we have that $l(a_s) = e$. This can only occur because  \eqref{eq:test}  held for $e$, and $e$ was the least such value for which it held. 
Hence if $\seq{\tau, Q} = p_{s_1}$ then $[\tau] \cap S_e= \emptyset$. This implies that \eqref{eq:test} will never again hold for $e$ and hence for all  $s > s_1$, $l(a_s) \ge e+1$.

If $l(a_{s+1}) > l(a_s)$, then condition \eqref{eq:test} does not hold. Hence as $\lim_s l(a_s) = \infty$, for infinitely many stages $s$, condition \eqref{eq:test} does not hold. This implies that $X$ has infinite length, hence is a Martin-L\"of random set, and $\rho(P \mid X) \leq 1/2$. Now assume that for some $e$, $X \in S_e$. Let $s_0$ be a stage such that for all $s \ge s_0$, $l(a_s) > e$. Let $\seq{\tau, Q} = p_{s_0}$. It must be that for any finite string $\nu$ such that $\tau \preceq \nu \prec X$, $\mu_{\nu}(S_e) > \epsilon(s_0,e)$ because for all $s \ge s_0$ we know that condition \eqref{eq:test} does not hold for $e$. Hence $\rho(S_e\mid X) > 0$. 
\end{proof}

%
%

\bibliographystyle{plain}
\bibliography{references}

\end{document}